\documentclass[11pt]{article}
\usepackage{graphicx} 

\usepackage{geometry}
\usepackage{inputenc}
\usepackage{enumerate}
\usepackage{amssymb}
\usepackage{amsmath}
\usepackage{mathrsfs}
\usepackage{amsthm}
\usepackage{tikz-cd}
\usepackage{tikz-cd,wrapfig,amscd}
\usepackage{setspace}
\usepackage{comment}
\usetikzlibrary{decorations.pathreplacing, calc}
\usepackage{hyperref} 

\theoremstyle{definition}
\newtheorem{definition}{Definition}[section]

\newtheorem{remark}[definition]{Remark}
\theoremstyle{plain}
\newtheorem{theorem}[definition]{Theorem}
\newtheorem{lemma}[definition]{Lemma}
\newtheorem{proposition}[definition]{Proposition}

\newtheorem{question}[definition]{Question}
\newtheorem{notation}[definition]{Notation}

\DeclareMathOperator{\supp}{\text{supp}}

\DeclareMathOperator{\ad}{\text{ad}}

\title{The quantitative coarse Baum--Connes conjecture for free products}
\author{Jintao Deng\thanks{Department of Mathematics, Shanghai University of Finance and Economics, Shanghai, China, 200433} \and Ryo Toyota\thanks{Department of Mathematics, Texas A\&M University, College Station, TX 77843, USA}}
\date{}

\begin{document}

\maketitle

 \begin{abstract}
    Let $X$ and $Y$ be uniformly locally finite metric spaces. In this paper, we prove the quantitative coarse Baum--Connes conjecture for the free product $X*Y$ under the assumption that the conjecture holds for both $X$ and $Y$.
 \end{abstract}

\onehalfspacing

\section{Introduction}

Higher index theory is a generalization of the Atiyah--Singer index theorem to noncompact settings via $K$-theory of $C^*$-algebras. One of the central problems in this area is the coarse Baum--Connes conjecture, which provides a framework for computing higher indices of elliptic operators on noncompact manifolds.
The coarse Baum--Connes conjecture asserts that a certain coarse assembly map 
$$\mu: KX_*(X)\to K_*(C^*(X))$$
is an isomorphism for a metric space $X$ with bounded geometry. This conjecture has many significant applications in analysis, geometry, and topology, including the Novikov conjecture on the homotopy invariance of higher signatures and the Gromov--Lawson conjecture on positive scalar curvature. It has been verified for a large class of metric spaces, including spaces with finite asymptotic dimension~\cite{Yu1998NovikoforFiniteAsymptoticDimension} and spaces admitting a coarse embedding into Hilbert space~\cite{Yu2000CoarseBaumConnesforSpacesUniformlyEmbeddable}.

For any two groups $G$ and $H$, the free product $G*H$ provides a canonical way to construct a new group from them. Several coarse geometric properties are known to be preserved under free products. 
Bell and Dranishnikov~\cite{BellDranishnikov01Asymptotic} showed that the asymptotic dimension of $G*H$ is controlled by that of $G$ and $H$. 
Chen, Dadarlat, Guentner, and Yu~\cite{ChenDadarlatGuentnerYu03Uniform} proved that $G*H$ has property~(A) (resp. admits a coarse embedding into Hilbert space) if and only if $G$ and $H$ do.

In this paper, we study the higher index theory for free products of a pair of uniformly locally finite metric spaces. As a generalization of free products for groups, Fukaya--Matsuka \cite{Fukaya-Matsuka:FreeProductOfcoarselyConvex} introduced the notion for metric spaces and proved that coarse convexity is preserved by free products. In \cite{Wang-Yao:Coarse-geometry-free-product}, Q. Wang and J. Yao studied the coarse geometry preserved by free product of metric spaces. 

All those large-scale properties are sufficient conditions for the coarse Baum--Connes conjecture. 
These results naturally lead to the following question regarding the coarse Baum--Connes conjecture:
\begin{question}\label{question}
Let $X$ and $Y$ be uniformly locally finite metric spaces. If $X$ and $Y$ satisfy the coarse Baum--Connes conjecture, does $X*Y$ also satisfy it?
\end{question}

For a finitely generated group $G$, the coarse Baum--Connes conjecture for the space $G$ equipped with a word metric is equivalent to the Baum--Connes conjecture for $G$ with coefficients in $\ell^{\infty}(G, \mathcal{K})$, where $\mathcal{K}$ is the algebra of compact operators on an infinite-dimensional separable Hilbert space. Regarding the Baum–Connes conjecture with coefficients, Oyono-Oyono~\cite{OO:BC-group-acting-on-trees} proved that for a group 
$G$ acting on a tree, the Baum--Connes conjecture with coefficients holds if and only if it holds for all vertex stabilizers. Tu~\cite{Tu:BCC-group-acting-on-trees} studied the conjecture for groups acting on trees under certain conditions by constructing gamma elements. However, the cases of groups admitting actions on trees typically rely on the Pimsner–Voiculescu six-term exact sequence, and this method can not be applied to general metric spaces.

To study Question \ref{question} for metric spaces, we consider the quantitative coarse Baum--Connes conjecture which is a stronger version of coarse Baum--Connes conjecture. The following is our main result.  

\begin{theorem}\label{main theorem}
    If $X$ and $Y$ are uniformly locally finite metric spaces that satisfy the quantitative coarse Baum--Connes conjecture, then their free product $X*Y$ also satisfies the quantitative coarse Baum--Connes conjecture.
\end{theorem}

To study the free product of metric spaces, we apply the quantitative $K$-theory introduced by Yu \cite{Yu1998NovikoforFiniteAsymptoticDimension}.
In ~\cite{Oyono-Oyono-Yu-On-Quantitative-K-theory, Oyono-Oyono-Yu2019QuantitativeK-theoryandKunnethformula}, Oyono-Oyono and Yu systematically developed quantitative $K$-theory. Guentner, Willett, and Yu~\cite{GuentnerWillettYu2024DynamicalComplexity} introduced dynamical complexity and applied those techniques to the Baum--Connes conjecture for dynamical systems. Since these approaches establish that assembly maps are isomorphisms via quantitative assembly maps, it is natural to formulate quantitative versions of the assembly map (cf. \cite[Section 6]{Oyono-Oyono-Yu-On-Quantitative-K-theory}, \cite{zhang2024QuantitativeCoarseBaumConnes}, and \cite{OO-Yu:2024Quantitative-Index}).

Our proof is inspired by the notion of geometric complexity (also known as decomposition complexity) introduced in~\cite{GuentnerTesseraYu2012ANotionofGeometricComplexity}. Let $\mathcal{X}$ and $\mathcal{Y}$ be families of metric spaces and let $r>0$. We say that $\mathcal{X}$ is $r$-decomposable over $\mathcal{Y}$ if for every $X\in\mathcal{X}$ there exists a decomposition $X=Y_0\cup Y_1$ such that each $Y_i$ admits a further decomposition
\begin{align*}
Y_i=\bigsqcup_j Y_{i,j},
\end{align*}
where each $Y_{i,j}\in\mathcal{Y}$ and $d(Y_{i,j},Y_{i,j'})>r$ for $j\neq j'$.

For a pair of metric spaces $X$ and $Y$, we exploit the geometric decomposition of the Bass--Serre tree $T$ associated with the free product $X*Y$. However, since $T$ typically has infinite valence, a subset of $X*Y$ corresponding to a bounded subset of $T$ consists of infinitely many copies of $X$ and $Y$ glued together in a highly intricate way. Another difficulty arises from the fact that the coarse Baum--Connes conjecture does not, in general, pass to subspaces. In \cite[Theorem~2.13]{OO-Yu:2024Quantitative-Index}, it is shown that the quantitative coarse Baum--Connes conjecture for a metric family can be recovered from a geometric decomposition provided that all subspaces of the pieces satisfy the conjecture uniformly.
Our main theorem does not assume that subspaces of $X$ or $Y$ satisfy the quantitative coarse Baum--Connes conjecture. To overcome this issue, we construct a decomposition of $X*Y$ that avoids such permanence properties. The detailed analysis in Lemmas~\ref{subtree} and~\ref{tree setminus tree}, which is specific to free products, plays a key role in the proof.

As an application of Theorem~\ref{main theorem}, we obtain new examples of metric spaces satisfying the coarse Baum--Connes conjecture. Arzhantseva and Tessera~\cite{ArzhantsevaTessera19GroupExtension} constructed a group $G$ which is an extension of groups coarsely embeddable into Hilbert space, but which itself does not admit a coarse embedding. It is known that such an extension satisfies the quantitative coarse Baum--Connes conjecture (Remark~\ref{generality of qcbc}). Let $X$ be any uniformly locally finite metric space satisfying the quantitative coarse Baum--Connes conjecture. By our main theorem, the free product $G*X$ also satisfies the quantitative coarse Baum--Connes conjecture, providing new examples of metric spaces satisfying the coarse Baum--Connes conjecture, especially in cases where $G$ or $X$ does not coarsely embed into Hilbert space.

This paper is organized as follows. In Section 2, we recall basic notions of Roe algebras and localization algebras. In Section 3, we develop techniques for computing quantitative $K$-theory and formulate the quantitative coarse Baum--Connes conjecture. In Section 4, we complete the proof of the main theorem.

\section{The Roe algebras and localization algebras}

In this section, we shall briefly recall the basic concepts of the coarse Baum--Connes conjecture.

Let $Z$ be a proper metric space. Recall that a metric space is said to be proper if every bounded subset is precompact. 

Let $\mathcal{H}$ be an infinite-dimensional, separable Hilbert space. We choose a countable dense subset $Z_0\subset Z$, and consider the Hilbert space $\ell^2(Z_0,\mathcal{H})$. Every element $T\in B(\ell^2(Z_0,\mathcal{H}))$ can be expressed as an $(Z_0\times Z_0)$-matrix $T=(T_{x,y})_{x,y\in Z_0}$. 

\begin{definition}
    Let $T=(T_{x,y})_{x,y\in Z_0}\in B(\ell^2(Z_0,\mathcal{H}))$.
    \begin{enumerate}[(1)]
        \item The propagation of $T$ is defined as  
$${\rm propagation}(T):=\sup\{d(x,y):T_{x,y}\neq 0\}.$$
 $T$ is said to have finite propagation if ${\rm porpagation}(T)<\infty$.
 \item $T$ is said to be locally compact, if for every compact subset $B\subset Z$, the operators $T\cdot\chi_{B}$ and $\chi_{B}\cdot T$ are compact operator on $\ell^2(Z_0,\mathcal{H})$, where $\chi_{B}$ is the multiplication operator by the characteristic function of $B\cap Z_0$.
    \end{enumerate}
\end{definition}

\begin{definition}
The algebraic Roe algebra $\mathbb{C}[Z]$ is defined to the $*$-subalgebra of $B(\ell^2(Z_0,\mathcal{H}))$ consisting of all locally compact operators with finite propagation. The Roe algebra $C^*(Z)$ is defined to be the $C^*$-algebra of the completion of $\mathbb{C}[Z]$ under the operator norm on $\ell^2(Z_0,\mathcal{H})$.
\end{definition} 

The definition of Roe algebras is independent on the choice of dense subset $Z_0\subset Z$ . Moreover, it is invariant under coarse equivalences (c.f.~\cite[Theorem 5.1.15]{Willett-Yu:Higher-index-theory}).  

\begin{definition}
    The algebraic localization algebra of $Z$, denoted by $\mathbb{C}_L[Z]$ is defined to the $*$-algebra of all uniformly bounded, and uniformly continuous functions $f: [0,\infty)\to C^*(Z)$ satisfying
    $$
    {\rm propagation}(f(t))\to 0,~\text{as}~t \to \infty.
    $$
    The localization algebra $C^*_L(Z)$ is defined to be the completion of $\mathbb{C}[Z]$ under the norm 
    $$
    \|f\|:=\sup_{t\in [0,\infty)}\|f(t)\|.
    $$
\end{definition}
The localization algebra was introduced by Yu~\cite{Yu1997LocalizationandCoarseBaumConnes} and he proved that the $K$-theory of the localization algebra $K_*(C^*_L(Z))$ is isomorphic to the $K$-homology groups of $K_*(Z)$ when $Z$ is a simplicial complex and it was strengthen to any locally compact metric space by Qiao and Roe  \cite{QiaoRoe10OnTheLocalization}. A key advantage of the localization algebra is that its $K$-theory is computable.
This is because the algebra encodes local geometric information in a way that
admits cutting-and-pasting arguments on the underlying space.
In particular, one can decompose the space into simpler pieces and recover
the global $K$-theory from these pieces.

There is a natural evaluation-at-zero map $e: C^*_L(Z)\to C^*(Z)$ defined by 
$$e(f)=f(0),~\forall~f\in C^*_L(Z).$$
This map is a $*$-homomorphism, and thus induces a homomorphism 
$$e_*: K_*(C^*_L(Z))\to K_*(C^*(Z))$$
on $K$-theory.

Next, we shall recall the coarse Baum--Connes conjecture for a metric space $X$. We denote by $B_X(x;r)$ the ball in $X$ of radius $r$ centered at $x$. Recall that a metric space is said to be uniformly locally finite if for every $r>0$, there exists a constant $N_r$ such that $\sup_{x\in X}\#B_X(x;r)\leq N_r$. 

Let $d>0$, the Rips complex is the simplicial complex whose vertex set is $X$ and a subset $\{x_1,x_2,\cdots,x_n\}\subset X$ spans a simplex if $d(x_i,x_j)\leq d$ for all $1\leq i,j\leq n$.

The Rips complex $P_d(X)$ is endowed with the spherical metric as follows. The simplex spanned by $\{x_1,x_2,\cdots,x_n\}$ is identified with the quadrant with diameter $1$:
$$S^{n-1}_+=\left\{(a_1,a_2,\cdots, a_n):\sqrt{\sum_{i=1}^na_i^2}=\frac{2}{\pi},~a_i\geq 0,\forall~1\leq i\leq n\right\}$$
via the map
$$
\sum_{j=1}^n a_jx_j\mapsto \frac{2}{\pi}\left(\frac{a_1}{\sum_{i=1}^n a_i^2},\cdots, \frac{a_n}{\sum_{i=1}^na_i^2} \right).
$$
Next, the Rips complex $P_s(X)$ is equipped with the path metric, and the distance between different connected components is defined to be infinity.

For each $k>0$, we choose the countable dense subset 
$$X_k=\left\{\sum_{i}a_ix_i\in P_k(X): \forall~i,~a_i\in \mathbb{Q}\right\}.$$ Notice that $X_k\subset X_{k'}$ if $k<k'$. For every $k>0$, we have an evaluation-at-zero
$$ev_0: C^*_L(P_k(X))\to C^*(P_k(X))$$ 
and it induces a homomorphism 
$$(ev_0)_*: K_*(C^*_L(P_k(X)))\to K_*(C^*(P_k(X)))$$ 
on $K$-theory. Passing to inductive limit, we obtain the coarse assembly map
$$
\mu: \lim_{k\to \infty}K_*(C^*_L(P_k(X)))\to \lim_{k\to \infty}K_*(C^*(P_k(X)))\cong K_*(C^*(X)).
$$

\noindent \textbf{The coarse Baum--Connes conjecture.} Let $X$ be a uniformly locally finite metric space. The coarse Baum--Connes conjecture claims that the coarse assembly map 
$$
\mu: \lim_{k\to \infty}K_*(C^*_L(P_k(X)))\to  K_*(C^*(X)).
$$
is an isomorphism. 

The coarse Baum--Connes conjecture was formulated by Roe \cite{Roe:coarse-cohomology-index-theory-on-Riem}, and it has been verified for a large class of metric spaces (c.f.~\cite{ Yu1998NovikoforFiniteAsymptoticDimension, Yu2000CoarseBaumConnesforSpacesUniformlyEmbeddable,DWYu:CBC-for-rel-expanders}). 

\begin{definition}
    Let $X$ be a uniformly locally finite metric space. The obstruction algebra is defined as 
    $$
    C^*_{L,0}(P_k(X))=\{f\in C_L^*(P_k(X)):f(0)=0\}.
    $$
\end{definition}
Note that the obstruction algebra is an ideal in the localization algebra, and we have following short exact sequence
$$
0\to C^*_{L,0}(P_k(X))\to C_L^*(P_k(X))\to C^*(P_k(X))\to 0,
$$
which gives rise to the six-term exact-sequence on $K$-theory. Therefore, the coarse Baum--Connes conjecture is equivalent to the vanishing of the following group
\begin{align*}
    \lim_{k\to \infty}K_*(C^*_{L,0}(P_k(X))).
\end{align*}

\section{Controlled $K$-theory and the quantitative Baum--Connes conjecture}

In this section, we recall some techniques in controlled $K$-theory
and the quantitative coarse Baum--Connes conjecture.
We then establish a key lemma for the proof of the main theorem,
which allows us to deduce the conjecture from suitable decompositions.

\subsection{Filtered $C^*$-algebras and controlled $K$-theory}
In this subsection, we recall the concepts of filtered $C^*$-algebras and controlled $K$-theory from \cite{GuentnerWillettYu2024DynamicalComplexity}. These tools provide a controlled Mayer--Vietoris argument for the quantitative coarse Baum--Connes conjecture.  

We first recall the concept of filtrations for $C^*$-algebras to encode geometry of metric spaces.
\begin{definition}
    A filtration on a $C^*$-algebra $A$ is a family of self-adjoint subspaces $(A_r)_{r\geq 0}$ of $A$ indexed by the non-negative real numbers $r$ with the following properties:
    \begin{enumerate}[(1)]
        \item if $r_1\leq r_2$, then $A_{r_1}\subset A_{r_2}$;

        \item for all $r_1,r_2$, we have $A_{r_1}\cdot A_{r_2}\subset A_{r_1+ r_2}$;

        \item the union $\bigcup_{r\geq 0} A_r$ is dense in $A$.
        
    \end{enumerate}
\end{definition}
In the case of Roe algebras, the filtration is defined by propagations. We apply this framework to the obstruction algebras $C^*_{L,0}(P_k(X))$. For $r>0$, we define
\begin{align*}
    C^*_{L,0}(M)_r:=\{f\in C^*_{L,0}(M):\text{propagation}(f(t))<r~\forall t\in [1,\infty)\}.
\end{align*}

Next, we recall the quantitative $K$-theory introduced by G. Yu \cite{Yu1998NovikoforFiniteAsymptoticDimension} to prove the coarse Baum-Connes conjecture for metric spaces with finite asymptotic dimension. 
\begin{definition}\label{quasiprojection}
Let $A$ be a $C^*$-algebra and $0<\varepsilon<\frac{1}{4}$. An $\varepsilon$-\emph{quasi-projection} in $A$ is an element $p \in A$ such that $p = p^*$ and $\|p^2 - p\| < \varepsilon$. 

If $S$ is a self-adjoint subspace of $A$, write $\mathrm{M}_n(S)$ for the matrices in $\mathrm{M}_n(A)$ with all entries in $S$, and let $P^{\varepsilon}_n(S)$ denote the set of $\varepsilon$-quasi-projections in $\mathrm{M}_n(S)$.

Let $\chi = \chi_{(1/2, \infty)}$ be the characteristic function of the interval $(1/2, \infty)$. Then $\chi$ is continuous on the spectrum of any $\varepsilon$-quasi-projection, so we obtain a map:
\[
\kappa : P^{\varepsilon}_n(S) \to P_n(A), \quad p \mapsto \chi(p),
\]
where $P_n(A)$ denotes the projections in $\mathrm{M}_n(A)$.
\end{definition}

\begin{definition}\label{quantitatike k-0}
Let $A$ be a non-unital $C^*$-algebra with a filtration $(A_r)_r$. Let $\Tilde{A}_r$ denote the subspace $A_r + \mathbb{C}1$ of $\Tilde{A}$.

Using the inclusions
\[
P^{\varepsilon}_n(\Tilde{A}_r) \ni p \mapsto
\begin{pmatrix}
p & 0 \\
0 & 0
\end{pmatrix}
\in P^{\varepsilon}_{n+1}(\Tilde{A}_r),
\]
we may define the union
\[
P^{\varepsilon}_{\infty}(\Tilde{A}_r) := \bigcup_{n=1}^\infty P^{\varepsilon}_n(\Tilde{A}_r).
\]

Let $C((0,1], \Tilde{A}_r)$ denote the self-adjoint subspace of the $C^*$-algebra $C((0,1], \Tilde{A})$ consisting of continuous functions with values in $\Tilde{A}_r$.

Let us assume that $0<\varepsilon<\frac{1}{24}$ and define an equivalence relation on $P^{\varepsilon}_{\infty}(\Tilde{A}_r) \times \mathbb{N}$ by declaring $(p, m) \sim (q, n)$ if there exists a positive integer $k$ and an element $h \in P^{6\varepsilon}_{\infty}(C([0,1], \Tilde{A}_{2r}))$ such that
\[
h(0) =
\begin{pmatrix}
p & 0 \\
0 & 1_{n+k}
\end{pmatrix}, \quad
h(1) =
\begin{pmatrix}
q & 0 \\
0 & 1_{m+k}
\end{pmatrix}.
\]
For $(p, m) \in P^{\varepsilon}_{\infty}(\Tilde{A}_r) \times \mathbb{N}$, denote by $[p, m]$ its equivalence class under $\sim$. The reason to choose $6\varepsilon$ in the equivalence class is to make the definition compatible with the quantitative Baum--Connes conjecture. (See Lemma~\ref{characterization of qcbc})

Let now $\rho: M_n(\Tilde{A}_r) \to M_n(\mathbb{C})$ be the restriction of the map induced on matrices by the canonical unital $*$-homomorphism $\rho: \Tilde{A}_r \to \mathbb{C}$ with kernel $A$.

Finally, we define
\[
K^{\varepsilon,r}_0(A) := \left\{ [p, m] \in P^{\varepsilon}_\infty(\Tilde{A}_r) \times \mathbb{N} \, \middle| \, \mathrm{rank}(\kappa(\rho(p))) = m \right\} / \sim.
\]
The set $K^{\varepsilon,r}_0(A)$ is equipped with an operation defined by
\[
[p, m] + [q, n] := \left[
\begin{pmatrix}
p & 0 \\
0 & q
\end{pmatrix}, m+n
\right].
\]
\end{definition}

\begin{definition}\label{quasiunitary}
    Let $A$ be a unital $C^*$-algebra and $0<\varepsilon<\frac{1}{4}$. An $\varepsilon$-\emph{quasi-unitary} in $A$ is an element $u \in A$ such that
\[
\|1 - uu^*\| < \varepsilon \quad \text{and} \quad \|1 - u^*u\| < \varepsilon.
\]

If $S$ is a self-adjoint subspace of $A$ containing the unit, we denote by $U^{\varepsilon}_n(S)$ the collection of quasi-unitaries in $M_n(S)$.

Note that since $\|1 - u^*u\| < \varepsilon < 1$, the element $u^*u$ is invertible. Hence, there is a well-defined map
\[
\kappa : U^{\varepsilon}_n(S) \to U_n(A), \quad u \mapsto u(u^*u)^{-1/2},
\]
where $U_n(A)$ denotes the group of unitaries in $M_n(A)$.

\end{definition}

\begin{definition}\label{quantitative k-1}
Let $A$ be a non-unital $C^*$-algebra with a filtration $(A_r)_r$.
Using the inclusions
\[
U^{\varepsilon}_n(\Tilde{A}_r) \ni u \mapsto
\begin{pmatrix}
u & 0 \\
0 & 1
\end{pmatrix}
\in U^{\varepsilon}_{n+1}(\Tilde{A}_r),
\]
we define the union
\[
U^{\varepsilon}_{\infty}(\Tilde{A}_r) := \bigcup_{n=1}^\infty U^{\varepsilon}_n(\Tilde{A}_r).
\]

We define an equivalence relation on $\mathcal{U}^{\varepsilon}_\infty(\Tilde{A}_r)$ by $u \sim v$ if there exists an element
\[
h \in U^{6\varepsilon}_\infty(C((0,1], \Tilde{A}_{2r}))
\quad \text{such that} \quad h(0) = u \quad \text{and} \quad h(1) = v.
\]
For $u \in U^{\varepsilon}_\infty(\Tilde{A}_r)$, denote by $[u]$ its equivalence class under $\sim$.
Finally, define
\[
K^{\varepsilon,r}_1(A) := U^{\varepsilon}_\infty(\Tilde{A}_r) \big/ \sim.
\]
The set $K^{\varepsilon}_1(A)$ is equipped with a binary operation defined by
\[
[u] + [v] := \left[
\begin{pmatrix}
u & 0 \\
0 & v
\end{pmatrix}
\right].
\]

\end{definition}
The advantage of quantitative $K$-theory is that it is more computable compared with the usual operator $K$-theory. It also has connections with the usual $K$-theory.  

\begin{definition}
Let $A$ be a non-unital $C^*$-algebra with a filtration. 
The comparison maps are defined as: 
\[
c_0 : K^{\varepsilon,r}_0(A) \longrightarrow K_0(A), \quad [p, m] \mapsto [\kappa(p)] - [1_m],
\]
\[
c_1 : K^{\varepsilon,r}_1(A) \longrightarrow K_1(A), \quad [u] \mapsto [\kappa(u)],
\]
\[
c := c_0 \oplus c_1 : K^{\varepsilon,r}_*(A) \longrightarrow K_*(A).
\]
\end{definition}

The map $c$ is asymptotically isomorphic in the following sense:

\begin{proposition}[\cite{GuentnerWillettYu2024DynamicalComplexity},~Proposition 4.9]
Let $A$ be a non-unital $C^*$-algebra with a filtration $(A_r)_r$, and $0<\varepsilon<\frac{1}{24}$. Then for any $x\in K_*(A)$, there exists $r>0$ and  $y\in K^{\varepsilon,r}(A)$ such that $c(y)=x$. Moreover, if $x,y \in K_*^{\varepsilon,s}(A)$ for some $s>0$ satisfies $c(x)=c(y)$, then there exists $r>0$ such that $x=y$ in $K_*^{\varepsilon,r}(A)$.
\end{proposition}

We fix some notations concerning the relationship between the control parameters $(\varepsilon,r)$ and the Rips parameter $k$ for obstruction algebras.
\begin{notation}\label{forgetful}
    Let $A$ be a filtered $C^*$-algebra. For two pairs $(\varepsilon,r)$ and $(\varepsilon',r')$ with $0<\varepsilon\leq \varepsilon'<\frac{1}{24}$ and $0<r\leq r'$, we have a natural forgetful map denoted by
    \begin{align*}
        \iota_0^{\varepsilon,\varepsilon',r,r'}&:K_0^{\varepsilon,r}(A)\to K_0^{\varepsilon',r'}(A) ,~ [p,m]\mapsto [p,m],\\
        \iota_1^{\varepsilon,\varepsilon',r,r'}&:K_1^{\varepsilon,r}(A)\to K_1^{\varepsilon',r'}(A) ,~ [u]\mapsto [u],\\
        \iota^{\varepsilon,\varepsilon',r,r'}:=\iota_0^{\varepsilon,\varepsilon',r,r'}\oplus \iota_1^{\varepsilon,\varepsilon',r,r'}&: K_*^{\varepsilon,r}(A) \to K_*^{\varepsilon',r'}(A).
    \end{align*}
\end{notation}

\begin{notation}\label{subspace inclusion}
    For an inclusion of metric spaces $i: Y\hookrightarrow X$, with countable dense subsets $Y_0\subset Y$ and $X_0\subset X$ satisfying $Y_0\subset X_0$, we denote the natural embedding by $V_i:\ell^2(Y_0)\otimes \mathcal{H}\to \ell^2(X_0)\otimes \mathcal{H}$. Then $\ad(V_i)$ does not increase propagation, and it induces a map
    \begin{align*}
        \ad(V_i)_*:K_*^{\varepsilon,r}(C^*_{L,0}(Y))\to K_*^{\varepsilon,r}(C^*_{L,0}(X)) 
    \end{align*}
    for every $0<\varepsilon<\frac{1}{24}$ and $r>0$. When no confusion arises, we omit $\ad(V_i)_*$ and regard elements of $K_*^{\varepsilon,r}(C^*_{L,0}(Y))$ as elements of $K_*^{\varepsilon,r}(C^*_{L,0}(X))$. In particular, for the inclusion $i:P_k(X)\hookrightarrow P_{k'}(X)$ between two Rips complexes of $X$, the induced map on their $K$-theory of the obstruction algebras are denoted by
    \begin{align*}
       \iota_{k,k'}:=\text{ad}(V_{i})_*:K^{\varepsilon,r}_*\left(C^*_{L,0}\left({P_k(X)}\right)\right) \longrightarrow K^{\varepsilon,r}_*\left(C^*_{L,0}\left({P_{k'}(X)}\right)\right).
    \end{align*}
\end{notation}

In the following, we shall recall some techniques for the computation of quantitative $K$-theory formulated in \cite{GuentnerWillettYu2024DynamicalComplexity}. Those techniques are crucial to reduce the computation of the quantitative $K$-theory of Roe algebras of the free product $G*H$ to that of $G$ and $H$.

\begin{definition}
 Let $A=(A_r)_{r\geq 0}$ be a filtered $C^*$-algebra, and $I$ be a $C^*$-ideal in $A$ equipped with its own filtration $I=(I_r)_{r\geq 0}$. We say that $I$ is a \textit{filtered ideal} of $A$ if for any $r \geq 0$, $I_r \subset A_r$, and if for any $r, s \geq 0$, $A_s \cdot I_r \cup I_r \cdot A_s \subset I_{s+r}$.
\end{definition}

\begin{definition}\label{definition of excisive}
Let $(I^\omega, J^\omega; A^\omega)_{\omega \in \Omega}$ be an indexed set of triples, where each $A^\omega$ is a filtered $C^*$-algebra, and each $I^\omega$ and $J^\omega$ is a filtered ideal in $A^\omega$. Give each stabilization $A^\omega \otimes \mathcal{K}$, $I^\omega \otimes \mathcal{K}$ and $J^\omega \otimes \mathcal{K}$ the filtrations 
\begin{align*}
    (A^\omega \otimes \mathcal{K})_r=A^\omega_r \otimes \mathcal{K},~(I^\omega \otimes \mathcal{K})_r=I^\omega_r \otimes \mathcal{K},~(J^\omega \otimes \mathcal{K})_r=J^\omega_r \otimes \mathcal{K}.
\end{align*}

\medskip

The collection $(I^\omega, J^\omega; A^\omega)_{\omega \in \Omega}$ of pairs of ideals and $C^*$-algebras containing them is \textit{uniformly excisive} if for any $r_0, m_0 \geq 0$ and $\varepsilon > 0$, there are $r \geq r_0$, $m \geq 0$, and $\delta > 0$ such that:

\begin{itemize}
    \item[(i)] for any $\omega \in \Omega$ and any $a \in (A^\omega \otimes \mathcal{K})_{r_0}$ of norm at most $m_0$, there exist elements $b \in (I^\omega \otimes \mathcal{K})_r$ and $c \in (J^\omega \otimes \mathcal{K})_r$ of norm at most $m$ such that $\|a - (b + c)\| < \varepsilon$;
    
    \item[(ii)] for any $\omega \in \Omega$ and any $a \in I^\omega \otimes \mathcal{K} \cap J^\omega \otimes \mathcal{K}$ such that
    \[
    d(a, (I^\omega \otimes \mathcal{K})_{r_0}) < \delta \quad \text{and} \quad d(a, (J^\omega \otimes \mathcal{K})_{r_0}) < \delta,
    \]
    there exists $b \in I^\omega_r \otimes \mathcal{K} \cap J^\omega_r \otimes \mathcal{K}$ such that $\|a - b\| < \varepsilon$.
\end{itemize}
\end{definition}

\begin{theorem}[\cite{GuentnerWillettYu2024DynamicalComplexity},~Proposition 7.7]\label{mayer-vietoris}
Let $\{(I^{\omega}, J^{\omega}; A^{\omega})\}_{\omega \in \Omega}$ be a uniformly excisive collection, where all algebras and ideals are non-unital. Then, for any $r \geq 0$, there exist $r_1, r_2 \geq r_0$ satisfying the following:

For each $\omega \in \Omega$ and each $x \in K^{\varepsilon,r}_*(A^\omega)$, there exists an element
\[
\partial_c (x) \in K^{\varepsilon,r_1}_*(I^{\omega} \cap J^{\omega})
\]
such that if $\partial_c (x)= 0$, then there exist elements
\[
y \in K^{\varepsilon,r_2}_*(I^{\omega}), \quad z \in K^{\varepsilon,r_2}_*(J^{\omega})
\]
such that
\[
\iota^{\varepsilon,\varepsilon,r,r_2}(x) = y + z \quad \text{in } K^{\varepsilon,r_2}_*(A^\omega).
\]
Furthermore, the boundary map $\partial_c$ has the following naturality:

Let $\{(K^{\theta}, L^{\theta}; B^{\theta})\}_{\theta \in \Theta}$ be another uniformly excisive collection of non-unital filtered $C^*$-algebras and ideals. Assume there exists a map $\pi : \Theta \to \Omega$ and, for each $\theta \in \Theta$, an inclusion
\[
A^{\pi(\theta)} \subset B^{\theta}
\]
such that for each $r \geq 0$:
\[
A^{\pi(\theta)}_r \subset B^{\theta}_r, \quad
I^{\pi(\theta)}_r \subset K^{\theta}_r, \quad
J^{\pi(\theta)}_r \subset L^{\theta}_r.
\]
Let $r$ be given, and let $r_1$ be as in the statement above for both uniformly excisive families.
Then the following diagram commutes:
\[
\begin{tikzcd}[column sep=large, row sep=large]
K^{\varepsilon,r}_*(A^{\pi(\theta)}) \arrow[r, "\partial_c"] \arrow[d] & K^{\varepsilon,r_1}_*(I^{\pi(\theta)} \cap J^{\pi(\theta)}) \arrow[d] \\
K^{\varepsilon,r}_*(B^{\theta}) \arrow[r, "\partial_c"] & K^{\varepsilon,r_1}_*(K^{\theta} \cap L^{\theta}),
\end{tikzcd}
\]
where the vertical maps are induced by subspace inclusions.
\end{theorem}

\subsection{Quantitative coarse Baum--Connes conjecture}
In this subsection, we recall the definition of the quantitative assembly map and the quantitative coarse Baum--Connes conjecture.

Since we have an isomorphism (c.f.~\cite[Lemma 6.3]{Toyota2025ControlledK-theory})
\begin{align*}
    K_*^{\varepsilon,r}(C^*_L(P_k(X)))\cong K_*(C^*_L(P_k(X))) 
\end{align*}
for all $0<\varepsilon<\frac{1}{8}$ and $r>0$, the evaluation at $t=0$ induces 
\begin{align*}
    \mu_{k,*}^{\varepsilon,r}:=(ev_0)_*:K_*(C^*_L(P_k(X)))\to K_*^{\varepsilon,r}(C^*(P_k(X))).
\end{align*}
This map is called the quantitative assembly map.
We define the following two statements:

\begin{enumerate}[i)]
    \item $QI(k,k',\varepsilon,r)$: for any $x\in K_*(C^*_L(P_k(X)))$ with $\mu_{k,*}^{\varepsilon,r}(x)=0\in K_*^{\varepsilon,r}(C^*(P_k(X)))$, we have $\iota_{k,k'}(x)=0$, where $\iota_{k,k'}$ is induced by the inclusion $P_k(X)\hookrightarrow P_{k'}(X)$ (See Notation~\ref{subspace inclusion}).

    \item $QS(k,k',\varepsilon,\varepsilon',r,r')$: for any $y\in K_*^{\varepsilon,r}(C^*(P_k(X)))$, there exists $x \in K_*(C^*_L(P_{k'}(X)))$ such that $\mu_{k',*}^{\varepsilon',r'}(x)=\iota^{\varepsilon,\varepsilon',r,r'}\circ \iota_{k,k'}(y)\in K_*^{\varepsilon',r'}(C^*(P_{k'}(X)))$.
\end{enumerate}
\begin{definition}
Let $\mathcal{X}$ be a family of metric spaces.
\begin{enumerate}[i)]
    \item We say that $\mathcal{X}$ satisfies the quantitative injectivity of the coarse assembly map if for any $k\in \mathbb{N}$, there exists $0<\varepsilon<\frac{1}{24}$ such that for any $r>0$, there exists $k'\geq k$ for which $QI(k,k',\varepsilon,r)$ for all $X\in \mathcal{X}$.

    \item We say that $\mathcal{X}$ satisfies the quantitative surjectivity of the coarse assembly map if for any $k\in \mathbb{N}$, there exists $0<\varepsilon<\frac{1}{24}$ such that for any $r>0$, there exists $k'>k$, $r'>r+1$ and $\frac{1}{24}>\varepsilon'>\varepsilon$ satisfying $QS(k,k',\varepsilon,\varepsilon',r,r')$ for all $X\in \mathcal{X}$.

    \item We say $\mathcal{X}$ satisfies the quantitative coarse Baum--Connes conjecture if it satisfies both the quantitative injectivity of the coarse assembly map and the quantitative surjectivity of the coarse assembly map.
\end{enumerate}
\end{definition}

The short exact sequence
\begin{align*}
0\to C^*_{L,0}(P_k(X))
\xrightarrow{i} C^*_L(P_k(X))
\xrightarrow{\pi} C^*(P_k(X))
\to 0
\end{align*}
induces a six-term $(\lambda,h)$-exact sequence
of $(\lambda_D,h_D)$-controlled morphisms
\cite[Theorem~4.7]{Oyono-Oyono-Yu-On-Quantitative-K-theory}.

More precisely, $\lambda,\lambda_D$ are constants, and $h:(0,\frac{1}{4\lambda})\to [1,\infty)$, $h_D:(0,\frac{1}{4\lambda_D})\to [1,\infty)$ are functions.
There are natural maps
\begin{align*}
i_*:\,&K_*^{\varepsilon,r}(C^*_{L,0}(P_k(X)))\to K_*^{\varepsilon,r}(C^*_L(P_k(X))),\\
\pi_*:\,&K_*^{\varepsilon,r}(C^*_L(P_k(X)))\to K_*^{\varepsilon,r}(C^*(P_k(X))),
\end{align*}
for all $0<\varepsilon<\frac{1}{4}$ and $r>0$,
and a controlled boundary map
\begin{align*}
\partial:\,&K_*^{\varepsilon,r}(C^*(P_k(X)))\to K_{*+1}^{\lambda_D\varepsilon,h_D(\varepsilon)r}(C^*_{L,0}(P_k(X)))
\end{align*}
for $0<\varepsilon<\frac{1}{4\lambda_D}$ and $r>0$.
These maps satisfy the following properties:
\begin{enumerate}[i)]
\item The compositions $\pi_*\circ i_*$, $\partial\circ \pi_*$, and $i_*\circ \partial$ are the zero map.

\item For any $0<\varepsilon<\frac{1}{4\max\{\lambda,\lambda_D\}}$ and $r>0$, if $y\in K_*^{\varepsilon,r}(C^*(P_k(X)))$ satisfies $\partial y=0$, then there exists $x\in K_*^{\lambda\varepsilon,h(\varepsilon)r}(C^*_L(P_k(X)))$ such that
\begin{align*}
\pi_*(x)=\iota^{\varepsilon,\lambda\varepsilon,r,h(\varepsilon)r}(y).
\end{align*}
The analogous statements hold for the other compositions.
\end{enumerate}

As a consequence, the quantitative coarse Baum--Connes conjecture reduces to quantitative estimates for the vanishing of the $K$-theory of the obstruction algebras.

 \begin{lemma}[Lemma 3.22, \cite{zhang2024QuantitativeCoarseBaumConnes}]
    Let $(\lambda,h)$ and $(\lambda_D,h_D)$ be control pairs as above. For a family of metric spaces $\mathcal{X}$, the quantitative coarse Baum--Connes conjecture holds if and only if for any $k\in \mathbb{N}$ there exists $0<\varepsilon_k<\frac{1}{4\max\{\lambda,\lambda_D\}}$ such that for any $r>0$, there exist $k'\geq k$, $\frac{1}{4\max\{\lambda,\lambda_D\}}>\varepsilon'\geq \varepsilon_k$ and $r'\geq r+1$ such that  
    \begin{equation}\label{controlled estimate of obstruction}
        \iota^{\varepsilon_k,\varepsilon',r,r'}\circ \iota_{k,k'}:K^{\varepsilon_k,r}_*(C^*_{L,0}(P_k(X))) \to K^{\varepsilon',r'}_*(C^*_{L,0}(P_{k'}(X)))
    \end{equation}
    is a zero map for all $X\in \mathcal{X}$.
\end{lemma}

Next, we strengthen the above result by showing that $\varepsilon'$ in
\eqref{controlled estimate of obstruction}
can be chosen to be equal to $\varepsilon_k$.

\begin{lemma}\label{characterization of qcbc}
    If a family $\mathcal{X}$ of metric spaces satisfies the quantitative coarse Baum--Connes conjecture, then for any $k\in \mathbb{N}$, there exists $0<\varepsilon_k<\frac{1}{24\max\{\lambda,\lambda_D\}}$ satisfying that for any $r>0$, there exists $k'\geq k$, and $r'\geq r+1$ such that  
    \begin{equation}
        \iota^{\varepsilon_k,\varepsilon_k,r,r'}\circ \iota_{k,k'}:K^{\varepsilon_k,r}_*(C^*_{L,0}(P_k(X))) \to K^{\varepsilon_k,r'}_*(C^*_{L,0}(P_{k'}(X)))
    \end{equation}
    is a zero map for all $X\in\mathcal{X}$.
\end{lemma}

\begin{proof}
We deal with the case $*=0$, and the case $*=1$ is the same. Note that if $p$ is a self-adjoint operator with $\|p^2-p\|\leq \varepsilon'$, then $\sigma(p)\subset [-\varepsilon',\varepsilon'] \cup [1-\varepsilon',1+\varepsilon']=:K$. 
    Choose a polynomial $f$ such that $|f(x)-\chi_{[\frac{1}{2},\infty)}(x)|\leq \frac{\varepsilon_k}{2}$ for all $x\in K$. Then functional calculus yields that
\[
p \mapsto f(p)
\]
sends $(\varepsilon',r')$-projections to $(\varepsilon_k,(r')^{\deg(f)})$- projections, and moreover
\[
\|p - f(p)\| \le \varepsilon'.
\]
 It follows that the linear homotopy between $p$ and $f(p)$ is a path of $(6\varepsilon_k,(r')^{\deg(f)})$-projections for any $(\varepsilon_k, r)$-projection $p$.
 Therefore, 
     \begin{align*}
        f_*\circ\iota^{\varepsilon_k,\varepsilon',r,r'}\circ \iota_{k,k'}:K^{\varepsilon_k,r}_0(C^*_{L,0}(P_k(X))) \to K^{\varepsilon_k,(r')^{\deg(f)}}_0(C^*_{L,0}(P_{k'}(X)));~ [p,m]\mapsto [f(p), m]
          \end{align*}
    coincide with the forgetful map
    \begin{align*}
        \iota^{\varepsilon_k,\varepsilon_k,r,(r')^{deg(f)}}\circ \iota_{k,k'}:K^{\varepsilon_k,r}_0(C^*_{L,0}(P_k(X))) \to K^{\varepsilon_k,(r')^{\deg(f)}}_0(C^*_{L,0}(P_{k'}(X)));~ [p,m]\mapsto [p, 1_m].
    \end{align*}
    (See the definition of the equivalence relation $\sim$ to define $K^{\varepsilon,r}$ in Definition~\ref{quantitatike k-0}.)
\end{proof}

\begin{remark}
We remark here that the quantitative coarse Baum--Connes conjecture is invariant under uniformly coarsely equivalences in the following sense.

Let $\mathcal{X}$ and $\mathcal{Y}$ be families of metric spaces. We say that $\mathcal{X}$ is uniformly coarsely equivalent to some of $\mathcal{Y}$ if there exist $C>0$ and two functions $\rho_{\pm}:\mathbb{R}\to\mathbb{R}$ with $\lim_{t\to\infty}\rho_{\pm}(t)\to\infty$ such that for any $X\in \mathcal{X}$ there exists $Y\in \mathcal{Y}$ and a coarse equivalence $f:X\to Y$ such that $N_C(f(X))=Y$ and
    \begin{align*}
        \rho_{-}(d_{X}(x,y))\leq d_{Y}(f(x),f(y))\leq  \rho_{+}(d_{X}(x,y)).
    \end{align*}
We call the triple $(\rho_-,~\rho_+,~ C)$ as the distortion of the coarse equivalence. 
    
In the case when the family $\mathcal{X}$ is uniformly coarsely equivalent to $\mathcal{Y}$, the family $\mathcal{X}$ satisfies the quantitative coarse Baum--Connes conjecture, if  $\mathcal{Y}$ satisfies it.
\end{remark}

\begin{remark}\label{generality of qcbc}
    The quantitative coarse Baum--Connes conjecture can be reformulated as follows \cite[Theorem~4.15]{zhang2024QuantitativeCoarseBaumConnes}. For a metric space $X$, let $\bigsqcup X$ denote the disjoint union of countably many copies of $X$. We define the metric on $\bigsqcup X$ such that the distance between points lying in different copies is infinite. Then $X$ satisfies the quantitative coarse Baum--Connes conjecture if and only if $\bigsqcup X$ satisfies the coarse Baum--Connes conjecture. In particular, this implies that the quantitative coarse Baum--Connes conjecture holds for all uniformly locally finite metric spaces $X$ which admits a coarse embedding into a Hilbert space. More generally, since the quantitative coarse Baum--Connes conjecture is equivalent to the coarse Baum--Connes conjecture with coefficient $\ell^{\infty}(\mathbb{N},\mathcal{K})$, it holds for all metric space $X$ which admits a coarse fibration 
    \begin{align*}
        N\hookrightarrow X \twoheadrightarrow Q
    \end{align*}
    with $N$ and $Q$ admitting a coarse embedding into a Hilbert space \cite[Theorem~1.3]{DengGuo2024TwistedRoeAlgebras}.
\end{remark}

\subsection{Decompositions of spaces and quantitative estimates}
In this subsection, we apply the framework of quantitative $K$-theory to geometric decompositions of spaces. This allows us to reduce the problem for spaces with the lower complexity. Similar ideas can be found in \cite{GuentnerWillettYu2024DynamicalComplexity} and \cite{DengToyota2025twisted}.

Fix $k_0\in \mathbb{N}$.
    Let $(X,d)$ be a uniformly locally finite metric space and assume that $\omega=(Y_0,Y_1)$ is a decomposition of $X$. (i.e. $Y_0$ and $Y_1$ are subspaces of $X$ with $X=Y_0\cup Y_1$.) Let $R_r$ be an increasing function of $r$ such that $R_r> (k_0+1)r$. We decompose $Y_i=\bigsqcup_j Y_i^{j}$ into $3R_r$-separated subsets.
Then for $i=0,1$ we have
\begin{align*}
    N_r(P_{k_0}(Y_i))=\bigsqcup_j N_r(P_{k_0}(Y_i^{j})),
\end{align*}
where $N_r$ denotes the $r$-neighborhood in $P_{k_0}(X)$, and the above disjoint union is $R_r$-separated. We define the new distance $d^{R,r}_i$ on $ N_r(P_{k_0}(Y_i))=\bigsqcup_j N_r(P_{k_0}(Y_i^{j}))$ by 
\begin{align*}
    d_i^{R,r}(x,y)=\left\{\begin{array}{ll}
        d(x,y) & x\in N_r(P_{k_0}(Y_i^{j})), ~ y\in N_r(P_{k_0}(Y_i^{j})) \text{ for some }j   \\
         \infty & \text{otherwise.}
    \end{array} \right.
\end{align*}
We denote by $P_{k_0}(Y_i)^{+r}$ the space $N_r(P_{k_0}(Y_i))$ equipped with the distance $d_i^{R,r}$.
We define filtered subspaces of $C^*_{L,0}(P_{k_0}(X))$ by 
\begin{equation*}
    \begin{aligned}
        B_{r}^{X,R,\omega}&:=C^*_{L,0}(P_{k_0}(Y_0)^{+r})_r+C^*_{L,0}(P_{k_0}(Y_1)^{+r})_r,\\
        I_{r}^{X,R,\omega}&:=C^*_{L,0}(P_{k_0}(Y_0)^{+r})_r,\\
        J_{r}^{X,R,\omega}&:=C^*_{L,0}(P_{k_0}(Y_1)^{+r})_r.
    \end{aligned}
\end{equation*}
Consider the $C^*$-subalgebra $B^{X,R,\omega}:=\overline{\bigcup_{r\geq 0}B_r^{X,R,\omega}}$ of $C^*_{L,0}(P_{s_0}(X))$ and two ideals $I^{X,R,\omega}:=\overline{\bigcup_{r\geq 0}I_r^{X,R,\omega}(X)}$, and $J^{X,R,\omega}:=\overline{\bigcup_{r\geq 0}J_r^{X,R,\omega}(X)}$ of $B^{X,R,\omega}$.

\begin{lemma}\label{R,w excisive}
The family $\left(I^{X,R,\omega},J^{X,R,\omega};B^{X,R,\omega}\right)_{X,R,\omega}$ is uniformly excisive, where $X$ runs over all uniformly locally finite metric spaces, $\omega$ runs over all decompositions $\omega=(Y_0,Y_1)$ of $X$ and $R$ runs over all increasing functions of $r$ satisfying $R_r> 3(k_0+1)r$.
\end{lemma}

\begin{proof}
For given $m_0, \varepsilon, r_0$ in Definition~\ref{definition of excisive}, we prove that $m=4m_0$, $\delta=\frac{1}{2}\varepsilon$ and $r=r_0+1$ work.

We first verify the first condition in Definition~\ref{definition of excisive}. Note that $P_{k_0}(X)\subset N_{1}(P_{k_0}(Y_0)) \cup N_{1}(P_{k_0}(Y_1))$. We define 
\begin{align*}
    Z_0&:=N_{r_0}(P_{k_0}(Y_0)) \setminus N_{r_0}(P_{k_0}(Y_1)),\\
    Z_1&:=N_{r_0}(P_{k_0}(Y_1)) \setminus N_{r_0}(P_{k_0}(Y_0)),\\
    Z_{0,1}&:=N_{r_0}(P_{k_0}(Y_0)) \cap N_{r_0}(P_{k_0}(Y_1)).
\end{align*}
Note that $d(Z_0, Z_1)> r_0$. For any $a\in B_{r_0}^{X,R,\omega}$, we define
    \begin{align*}
        a_0&:=\chi_{Z_0} a \chi_{Z_0} + \chi_{Z_0} a \chi_{Z_{0,1}}+ \chi_{Z_{0,1}} a \chi_{Z_0} + \chi_{Z_{0,1}} a \chi_{Z_{0,1}},\\
        a_1&:=\chi_{Z_1} a \chi_{Z_1} + \chi_{Z_1} a \chi_{Z_{0,1}}+ \chi_{Z_{0,1}} a \chi_{Z_1}.
    \end{align*}
    Since $a$ has propagation less than $r_0$, we have $a=a_0+a_1$. In addition, $a_0$ and $a_1$ have propagation less than $r_0$, $a_0$ is supported on $N_{r_0}(P_{k_0}(Y_0))\times N_{r_0}(P_{k_0}(Y_0))$ and $a_1$ is supported on $N_{r_0}(P_{k_0}(Y_1))\times N_{r_0}(P_{k_0}(Y_1))$. Moreover we have
    \begin{align*}
        \|a_0\|\leq 4\|a\|,~\|a_1\|\leq 3\|a\|.
    \end{align*}

    Next, we prove the second condition. Let $\{\mu_0, \mu_1\}$ be a partition of unity subordinate to the cover $P_{k_0}(Y_0)^{+1}\cup P_{k_0}(Y_1)^{+1}$. Let $a\in I^{X,R,\omega}\otimes \mathcal{K}\cap J^{X,R,\omega}\otimes \mathcal{K}$ such that $\|a-a_0\|\leq \delta$ and $\|a-a_1\|\leq \delta$ for some $a_0\in \left(I^{X,R,\omega}\right)_{r_0}$ and $a_1\in \left(J^{X,R,\omega}\right)_{r_0}$. Define
    \begin{align*}
        b:=a_0\mu_1+a_1\mu_0.
    \end{align*}
    Then since the multiplication by $\mu_1$ does not increase the support, 
    \begin{align*}
        \supp{(a_0\mu_1)}\subset \supp(a_0)\subset N_{r_0}(P_{k_0}(Y_0))\times N_{r_0}(P_{k_0}(Y_0)).
    \end{align*}
   Since $\mu_1$ is supported on $N_1(P_{k_0}(Y_1))$, and $a_0$ has propagation less than $r_0$, we obtain that
    \begin{align*}
        \supp{(a_0\mu_1)}\subset N_{r_0+1}(Y_1) \times N_{1}(Y_1)\subset N_{r_0+1}(Y_1) \times N_{r_0+1}(Y_1) 
    \end{align*}
    By arguing similarly for the support of $a_1\mu_0$, we can show that 
\begin{align*}
    \supp(b) \subset \left( N_{r_0+1}(Y_0) \times N_{r_0+1}(Y_0)\right)\cap\left( N_{r_0+1}(Y_1) \times N_{r_0+1}(Y_1)\right)
\end{align*}
    and so $b\in I^{X,R,\omega}_{r_0+1}\cap J^{X,R,\omega}_{r_0+1}$.
    Moreover, we have 
    $$\|a-b\|=\|(a-a_0)\mu_1+(a-a_1)\mu_0\|\leq 2 \delta=\varepsilon.$$
\end{proof}

In the following, we define another triple $(I^{X,R,\omega,k},J^{X,R,\omega,k};B^{X,R,\omega,k})$ consisting of a filtered $C^*$-algebra and its ideals, which encodes increments of the parameter $k$ of the Rips complex.

For $k\geq k_0$, we define
\begin{align*}
    B_r^{X,R,\omega,k}&:=C^*_{L,0}(P_{k_0}(Y_0)^{+r})_r+C^*_{L,0}(P_{k_0}(Y_1)^{+r})_r +C^*_{L,0}\left(P_{k}\left(\bigsqcup_j Y_0^{(j)}\right)^{+r}\cap P_{k}\left(\bigsqcup_j Y_1^{(j)}\right)^{+r}\right)_{kr},\\
        I_r^{X,R,\omega,k}&:=C^*_{L,0}(P_{k_0}(Y_0)^{+r})_r +C^*_{L,0}\left(P_{k}\left(\bigsqcup_j Y_0^{(j)}\right)^{+r}\cap P_{k}\left(\bigsqcup_j Y_1^{(j)}\right)^{+r}\right)_{kr},\\
        J_r^{X,R,\omega,k}&:=C^*_{L,0}(P_{k_0}(Y_1)^{+r})_r +C^*_{L,0}\left(P_{k}\left(\bigsqcup_j Y_0^{(j)}\right)^{+r}\cap P_{k}\left(\bigsqcup_j Y_1^{(j)}\right)^{+r}\right)_{kr},
\end{align*}
where the Rips complex $P_{k}\left(\bigsqcup_j Y_i^{(j)}\right)$ is taken in terms of the separated coarse disjoint union $\bigsqcup_j Y_i^{(j)}$ and the $r$-neighborhood $P_{k}\left(\bigsqcup_j Y_i^{(j)}\right)^{+r}$ is taken relative to $P_{k_0}(X)$, i.e.
\begin{equation}\label{relative expansion}
    P_{k}\left(\bigsqcup_j Y_0^{(j)}\right)^{+r}=P_{k}\left(\bigsqcup_j Y_0^{(j)}\right)\cup \left\{p\in P_{k_0}(X);d\left(p,P_k\left(\bigsqcup_j Y_0^{(j)}\right)\cap P_{k_0}(X)\right)\leq r\right\}.
\end{equation} 
Then define a $C^*$-algebras $B^{X,R,\omega,k}:=\overline{\bigcup_{r\geq 0}B_r^{X,R,\omega,k}}$ and two ideals $I^{X,R,\omega,k}:=\overline{\bigcup_{r\geq 0}I_r^{X,R,\omega,k}(X)}$ and $J^{X,R,\omega,k}:=\overline{\bigcup_{r\geq 0}J_r^{X,R,\omega,k}(X)}$ of $B^{X,R,\omega,k}$.

By decomposing the space into 
\begin{align*}
    Z_0&:=P_{k_0}(Y_0)^{+r}\setminus\left( P_{k}\left(\bigsqcup_j Y_0^{(j)}\right)^{+r}\cap P_{k}\left(\bigsqcup_j Y_1^{(j)}\right)^{+r}\right)\\
    Z_1&:=P_{k_0}(Y_1)^{+r}\setminus \left(P_{k}\left(\bigsqcup_j Y_0^{(j)}\right)^{+r}\cap P_{k}\left(\bigsqcup_j Y_1^{(j)}\right)^{+r}\right)\\
    Z_{0,1}&:=P_{k}\left(\bigsqcup_j Y_0^{(j)}\right)^{+r}\cap P_{k}\left(\bigsqcup_j Y_1^{(j)}\right)^{+r},
\end{align*}
the same argument as Lemma~\ref{R,w excisive} shows the following lemma.
\begin{lemma}\label{k,w,R-excisive pair}
The family $\left(I^{X,R,\omega,k},J^{X,R,\omega,k};B^{X,R,\omega,k}\right)_{X,R,\omega,k}$ is uniformly excisive, where $X$ runs over all uniformly locally finite metric spaces, $R$ runs over all increasing functions of $r$ satisfying $R_r> 3(k_0+1)r$, $\omega$ runs over all decompositions $\omega=(Y_0,Y_1)$ of $X$ and $k$ runs over all numbers $k\geq k_0$.
\end{lemma}

For each $r$, we regard the value $r_2$ obtained from Theorem~\ref{mayer-vietoris} for the excisive pair in Lemma~\ref{k,w,R-excisive pair} as a function $r_2(r)$ of $r$, and define $R_r:=3(k_0+1)r_2(r)$. 

\begin{lemma}\label{one step decomposition}
Let $k_0 > 0$ and $r > 0$. Let $\mathcal{X}$ be a family of uniformly locally finite metric spaces. Suppose each $X \in \mathcal{X}$ admits a decomposition $X = Y_0(X) \cup Y_1(X)$ where each $Y_i(X)$ is a $3R_r$-separated disjoint union
$$
Y_i(X) = \bigsqcup_j Y_i(X)^j.
$$
Assume that for $i = 0,1$ and all $s \leq R_r$, the families
$$
\{Y_i(X)^j\}_{X,j} \quad\text{and}\quad \{N_s(Y_0(X)^j) \cap N_s(Y_1(X)^{j'})\}_{X,j,j'}
$$
satisfy the quantitative coarse Baum--Connes conjecture.

Then there exist $k\geq k_0$ and $r'\geq r$ such that the map
$$
K_*^{\varepsilon,r}(C^*_{L,0}(P_{k_0}(X))) \longrightarrow K_*^{\varepsilon,r'}(C^*_{L,0}(P_k(X)))
$$
is the zero map for sufficiently small $\varepsilon$ and for every $X \in \mathcal{X}$.
\end{lemma}

\begin{proof}
Let $x\in K^{\varepsilon,r}_*(C^*_{L.0}(P_{k_0}(X)))$. Since 
    \begin{align*}
        C^*_{L,0}(P_{k_0}(X))_{r}\subset C^*_{L,0}(P_{k_0}(Y_0)^{+r_0})_{r} + C^*_{L,0}(P_{k_0}(Y_1)^{+r_0})_{r}=B_{r}^{X,R,\omega},
    \end{align*}
    the element $x$ can be viewed as an element in $K_*^{\varepsilon,r}(B^{X,R,\omega})$. By Theorem~\ref{mayer-vietoris}, there exists $r_1\geq r$ and the boundary maps $\partial$ which fit in the commutative diagram below for every $k_1\geq k_0$:
    \[ \begin{tikzcd}
K_*^{\varepsilon,r}(B^{X,R,\omega}) \arrow{r}{\partial} \arrow{d} &K_*^{\varepsilon,r_1}(I^{X,R,\omega}\cap J^{X,R,\omega}) \arrow{d} \\
K_*^{\varepsilon,r}(B^{X,R,\omega,k_1}) \arrow{r}{\partial}& K_*^{\varepsilon,(k_0+1)r_1}(I^{X,R,\omega,k_1}\cap J^{X,R,\omega,k_1}).
\end{tikzcd}
\]
We analyze the right vertical map.
Using
\begin{align*}
P_{k_0}(N_r(Y_i))\subset P_{k_0}(Y_i)^{+r}\subset
P_{k_0}(N_{(k_0+1)r}(Y_i)),
\end{align*}
it factors as a composition of
forgetful maps and inclusion-induced maps: (Notation~\ref{forgetful} and Notation~\ref{subspace inclusion}):
\begin{align*}
    K_*^{\varepsilon,r_1}(I^{X,R,\omega}\cap J^{X,R,\omega})&=K_*^{\varepsilon,r_1}\left(C^*_{L,0}\left(P_{k_0}(Y_0(X))^{+r_1}\cap P_{k_0}(Y_1(X))^{+r_1}\right) \right)\\
    &\to K_*^{\varepsilon,r_1}\left(C^*_{L,0}(P_{k_0}(N_{(k_0+1)r_1}(Y_0(X))\cap N_{(k_0+1)r_1}(Y_1(X))))\right)\\
    &\to K_*^{\varepsilon,k_1r_1}\left(C^*_{L,0}\left(P_{k_1}\left(\bigsqcup_{j,j'} N_{(k_0+1)r_1}(Y_0(X)^{(j)})\cap N_{(k_0+1)r_1}(Y_1(X)^{(j')})\right)\right)\right)\\
    &\to K_*^{\varepsilon,(k_0+1)k_1r_1}\left(C^*_{L,0}\left(P_{k_1}\left(\bigsqcup_{j,j'} N_{(k_0+1)r_1}(Y_0(X)^{(j)})\cap N_{(k_0+1)r_1}(Y_1(X)^{(j')})\right)\right)\right)\\
    &\to K_*^{\varepsilon,(k_0+1)r_1}\left(I^{X,R,\omega,k_1}\cap J^{X,R,\omega,k_1}\right),
\end{align*}
and this is the $0$-map for sufficiently large $k_1$ which is independent of $X\in \mathcal{X}$ and $j,j'$. By the controlled exactness, there exist $y\in K_*^{\varepsilon,r_2}(I^{X,R,\omega,k_1})$ and $z\in K_*^{\varepsilon,r_2}(J^{X,R,\omega,k_1})$ such that $x=y+z \in K_*^{\varepsilon,r_2}(B^{X,R,\omega,k_1})$. Since 
\begin{align*}
    I^{X,R,\omega,k_1}_{r_2}\subset C^*_{L,0}\left(P_{k_1}\left(\bigsqcup_j Y_0^{(j)}(X)\right)^{+r_2}\right)_{k_1r_2}&\subset C^*_{L,0}\left(P_{k_1}\left(\bigsqcup_j N_{(k_0+1)r_2}(Y_0^{(j)}(X))\right)\right)_{k_1r_2}\\
    &=C^*_{L,0}\left(\bigsqcup_j P_{k_1}\left( N_{(k_0+1)r_2}(Y_0^{(j)}(X))\right)\right)_{k_1r_2},
\end{align*} 
by the definition \eqref{relative expansion} of the relative $r_2$-neighborhood of $P_{k_1}\left(\bigsqcup_j Y_0^{(j)}(X)\right)$
and $\{Y_{0,j}(X)\}$ satisfies the quantitative coarse Baum--Connes conjecture, we have $y=0$ in $K_*^{\varepsilon,r'}(C^*_{L,0}(P_k(X)))$ for sufficiently large $r'$ and $k$ which are independent of $X\in \mathcal{X}$. Also, the same thing holds for $z$.
Therefore, we have
\begin{align*}
    x=y+z=0 \in K_*^{\varepsilon,r'}\left(C^*_{L,0}\left( P_{k}\left( X\right)\right)\right).
\end{align*}
\end{proof}

This result plays a crucial role in the proof of our main result. Next, we shall establish decompositions for the free product $X \ast Y$ into neighborhoods of cosets of $X$ and $Y$. Using the above theorem, the quantitative coarse Baum–Connes conjecture for $X\ast Y$ reduces to that of $X$ and $Y$.

\section{Proof of the main theorem}

In this section, we complete the proof of Theorem~\ref{main theorem} by applying Theorem~\ref{one step decomposition} to suitable decompositions of free product. We briefly recall the definition of free product of metric spaces introduced in \cite{Fukaya-Matsuka:FreeProductOfcoarselyConvex}, focusing on the discrete case. Let $(X,d_X)$ and $(Y,d_Y)$ be uniformly locally finite metric spaces and we fix their base point $e_X\in X$ and $e_Y\in Y$. We denote by $\ell$ the length on $X\sqcup Y$ defined as
\begin{align*}
    \ell(v):=\left\{
    \begin{array}{cc}
      d_X(v,e_X),   &  \text{ if }v\in X,\\
      d_Y(v,e_Y),   & \text{ if } v\in Y.
    \end{array}
    \right.
\end{align*}
A word on $X\sqcup Y$ is a finite sequence
\begin{align*}
    w=w_1w_2\cdots w_n \quad \text{(} n\geq 0, w_i\in X\sqcup Y \text{)}
\end{align*}
such that 
\begin{align*}
    w_i\in X \Rightarrow w_{i+1}\in Y\\
    w_i\in Y \Rightarrow w_{i+1}\in X,
\end{align*}
where for $n=0$ we define $w=\phi$ to be the empty word. The free product $X*Y$ is the set of all words on $X\sqcup Y$ equipped with the distance $d$ defined as follows. Let $w=uw_1w_2\cdots w_m$ and $z=uz_1z_2\cdots z_n$ be any words where $u$ is their maximal common prefix.
If $w_1,z_1\in X$, or $w_1, z_1\in Y$, then we define their distance to be
\begin{align*}
d(w,z):=\left\{
\begin{array}{cc}
     \ell(w_m)+\sum\limits_{i=2}^{m-1} (1+\ell(w_i))+d_X(w_1,z_1)+\sum\limits_{i=2}^{n-1}(\ell(z_i)+1)+\ell(z_n)& \text{ if } w_1,z_1\in X\\
     \ell(w_m)+\sum\limits_{i=2}^{m-1} (1+\ell(w_i))+d_Y(w_1,z_1)+\sum\limits_{i=2}^{n-1}(\ell(z_i)+1)+\ell(z_n) & \text{ if } w_1,z_1\in Y
\end{array}    
\right.
\end{align*}
If $w_1\in X$ and $z_1\in Y$, then we define their distance to be
\begin{align*}
    d(w,z):=\ell(w_m)+\sum_{i=1}^{m-1} (1+\ell(w_i))+\sum_{i=1}^{n-1}(\ell(z_i)+1)+\ell(z_n).
\end{align*}
Note that if $(X,d_X)$ and $(Y,d_Y)$ are uniformly locally finite, then so is their free product.

We first construct a map from the free product to the Bass--Serre tree, and then consider decompositions of the Bass--Serre tree associated to the free product $X*Y$. However, as already noted in the introduction, the Bass--Serre tree $T$ has infinite vertex degree. Consequently, the preimage under this map of a bounded subset of $T$ consists of infinitely many copies of $X$ and $Y$ glued together in a highly intricate manner. An additional difficulty arises because the coarse Baum–Connes conjecture does not pass to subspaces. To overcome this, we decompose the free product directly, without relying on any information about subspaces.

We now recall the construction of the Bass--Serre tree $T$ associated with the free product $X*Y$. The vertex set of $T$ consists of the empty word $\phi$ and the following copies of $X$ and $Y$ in $X*Y$:
\begin{align*}
    \{wX:w \text{ is a word of }X\sqcup Y\text{ whose last letter does not belong to}X\},\\
    \{wY:w \text{ is a word of }X\sqcup Y\text{ whose last letter does not belong to}Y\}.
\end{align*}
The edge structure of $T$ is defined as follows:

\begin{enumerate}[i)]
    \item the empty word $\phi$ is connected to $\phi X$ and $\phi Y$,

    \item for a word $w\in X*Y$, whose last letter is in $Y$, $wX$ is connected to $wxY$ for all $x\in X$, and

    \item for a word $w\in X*Y$, whose last letter is in $X$, $wY$ is connected to $wyX$ for all $y\in Y$.
\end{enumerate}
Equipped with this graph structure, $T$ becomes a tree. (possibly with infinite degree), and we equip $T$ with the usual edge distance $d_T$.

Now we define a map $\pi:X*Y \to T$ by 
\begin{align*}
    \pi(w)=\left\{\begin{array}{cc}
       \phi,  & \text{if }w=\phi,  \\
        w'X, & \text{ if }w=w'x \text{ for a  word }w' \text{ and }x\in X,\\
        w'Y, & \text{ if }w=w'y \text{ for a  word }w' \text{ and }y\in Y.
    \end{array}
    \right.
\end{align*}
It is easy to see that $\pi$ is a contraction.

For any $n\in \mathbb{N}$, and $w\in X*Y$, we define $T^w_{n}$ as follows. If $w=w'x$ for some $x\in X$, then
\begin{align*}
    T^w_{n}:=\bigcup_{k=0}^{\lfloor \frac{n}{2}\rfloor}\{ w'x_1y_1\cdots x_{k}y_kX: x_i\in X,~y_i\in Y\} \cup \bigcup_{k=0}^{\lfloor \frac{n-1}{2}\rfloor}\{ w'x_1y_1\cdots x_{k}y_kx_{k+1}Y: x_i\in X,~y_i\in Y\}
\end{align*}
If $w=w'h$ for some $h\in H$, then  
\begin{align*}
    T^w_{n}:=\bigcup_{k=0}^{\lfloor \frac{n}{2}\rfloor}\{ w'y_1x_1\cdots y_{k}x_kY: x_i\in X,~y_i\in Y\} \cup \bigcup_{k=0}^{\lfloor \frac{n-1}{2}\rfloor}\{ w'y_1x_1\cdots y_{k}x_ky_{k+1}X: x_i\in X,~y_i\in Y\}
\end{align*}
The Figure~\ref{fig:subtree_of_depth n} is $T^w_n$ for $w=w'x$.
\begin{figure}[h]
\centering
\includegraphics[width=0.8\textwidth]{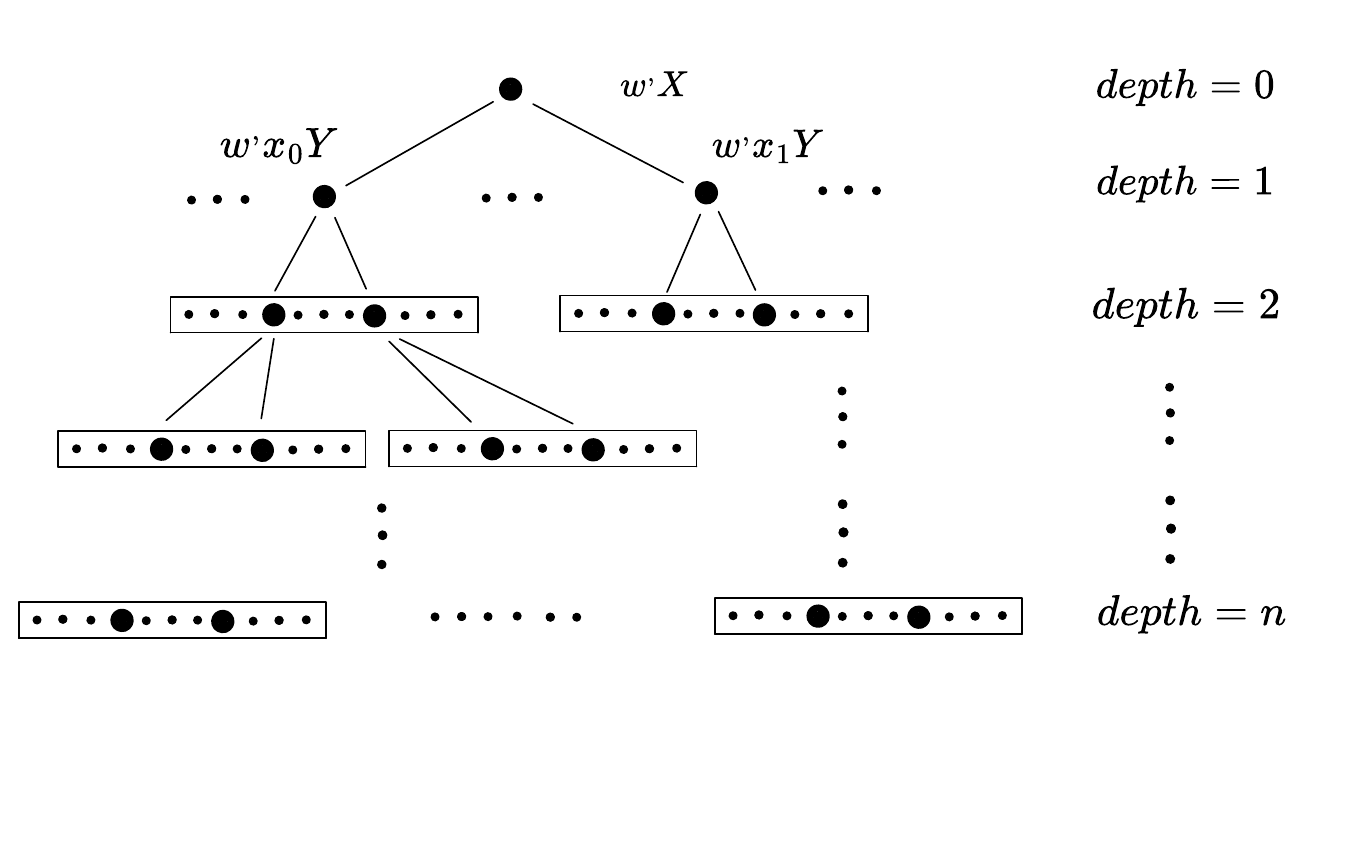}
\vspace{-15mm}
\caption{The subtree of depth $n$ rooted at $w'X$.}
\label{fig:subtree_of_depth n}
\end{figure}

\begin{lemma}\label{subtree}
Let $W\subseteq X*Y$ be any subset.
    If $X$ and $Y$ satisfy the quantitative coarse Baum--Connes conjecture, then so does the family $\{\pi^{-1}(T_n^w)\}_{w\in W}$ for every $n\in \mathbb{N}$.
\end{lemma}

\begin{proof}
    We argue by induction on $n$. When $n=0$, the statement is trivial since $\pi^{-1}(T_n^{w})$ is either $X$ or $Y$. Assume that the statement holds for $n$. We prove it for $n+1$.
    Set
    \begin{align*}
        S_1^w&:=\{v\in T_n^w:d_T(v,\pi(w))\leq n\},\\
        S_2^w&:=\{v\in T_n^w:d_T(v,\pi(w))= n+1\}.
    \end{align*}
    It suffices to show that for every $r_0>0$ and $k_0>0$, there exists $r\geq r_0$ and $k\geq k_0$ such that the map
    \begin{align*}
        K^{\varepsilon,r_0}_{*}\left(C^*_{L,0}\left( P_{k_0}(\pi^{-1}(T_n^w))\right)\right) \to K^{\varepsilon,r}_{*}\left(C^*_{L,0}\left( P_{k}(\pi^{-1}(T_n^w))\right)\right)
    \end{align*}
    is the zero map for all $w\in W$. Denote $R=3R_{r_0}$ and $A_1^{w}:=N_R(\pi^{-1}(S_1^{w}))$ and $A_2^{w}:=\pi^{-1}(T_{n+1}^{w})\setminus A_1^w$, where $N_R(\pi^{-1}(S_1^{w}))$ is an $r$-neighborhood of $\pi^{-1}(S_1^{w})$ in $\pi^{-1}(T_n^w)$ (see Figure~\ref{fig:decomposition_of_subtree}).
    Observe that $A_1^{w}$ is coarsely equivalent to $\pi^{-1}(S_1^{w})$ with the distortion independent of $w$  and thus the family $\{A_1^w\}_w$ satisfies the quantitative coarse Baum--Connes conjecture by inductive assumption.
    Let $Z$ denote $X$ or $Y$ according to the parity of $n$ and the last letter of $w$. Then $A_2^w$ can be written as a disjoint union $\bigsqcup_{w'} w'Z\setminus  B_{w'Z}(w'e_Z;R)$, which is $2R$-separated by the definition of the distance $d$ on the free product. If we denote by $A_{2,w'}^w:=w'Z\setminus  B_{w'Z}(w'e_Z;R)$, then $A_2^{w}=\bigsqcup A^w_{2,w'}$ is a $2R$-separated disjoint union of a metric family $\{A^w_{2,w'}\}_{w'}$. Each $A_{2,w'}^w$ is coarsely equivalent to $Z$ with the distortion independent of $w$ and $w'$,  and thus the family $\{A_{2,w'}^w\}_{w,w'}$ satisfies the quantitative coarse Baum--Connes conjecture.
     For $s<R_{r_0}$, we consider the metric family $\{N_{s}(A_1^w)\cap N_{s}(A^w_{2,w'})\}_{w,w'}$. Each $N_{s}(A_1^w)\cap N_{s}(A^w_{2,w'})$ is an annulus in $Z$ of radius between $R-s$ and $R+s$. Thus the metric family $\{N_{s}(A_1^w)\cap N_{s}(A^w_{2,w'})\}_{w,w'}$ is uniformly bounded and thus satisfies the quantitative coarse Baum--Connes conjecture. As a result, the lemma follows from Lemma~\ref{one step decomposition}.
\end{proof}

\begin{figure}[h]
\centering
\includegraphics[width=0.8\textwidth]{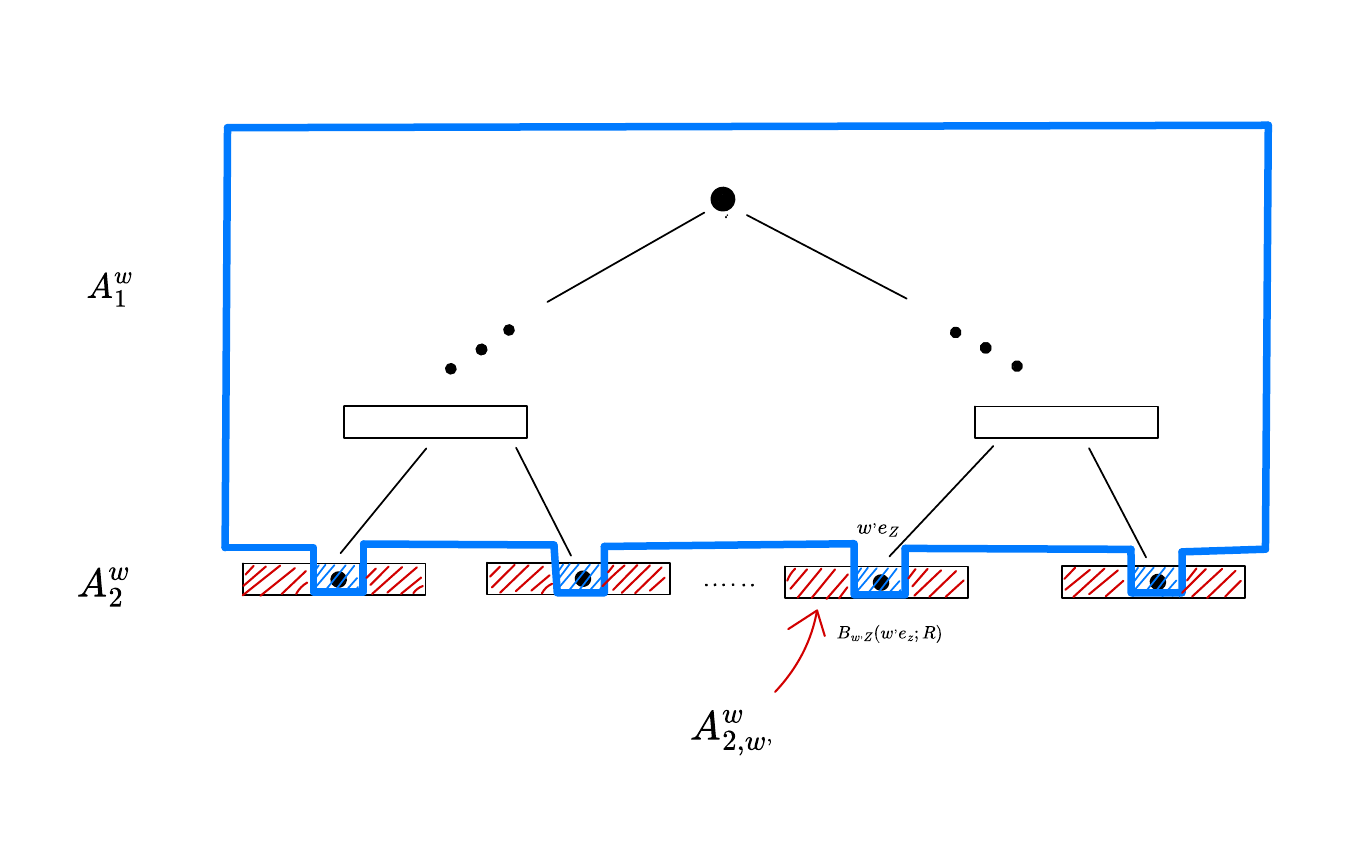}
\vspace{-10mm}
\caption{The decomposition $\pi^{-1}(T_n^w)=A_1^w\cup A_2^w$.}
\label{fig:decomposition_of_subtree}
\end{figure}

\begin{lemma}\label{tree setminus tree}
    Let $m,n$ be non-negative integers and $W\subseteq X*Y$. For each $w\in W$, we denote $S_{m,n}^{w}:=T_{m+n}^w\setminus T_n^w$. Then $\{\pi^{-1}(S_{m,n}^w)\}_{w\in W}$ satisfies the quantitative coarse Baum--Connes conjecture.
\end{lemma}
\begin{proof}
By induction on $n$, we prove that for every $m\in \mathbb{N}$, $r_0>0$ and $k_0>0$, there exist $r\geq r_0$ and $k\geq k_0$ such that the map
    \begin{align*}
        K^{\varepsilon,r_0}_{*}\left(C^*_{L,0}\left( P_{k_0}(\pi^{-1}(S_{m,n}^w))\right)\right) \to K^{\varepsilon,r}_{*}\left(C^*_{L,0}\left( P_{k}(\pi^{-1}(S_{m,n}^w))\right)\right)
    \end{align*}
    is the zero map, for all $w\in W$ for sufficiently small $\varepsilon>0$.
    Without loss of generality, we may assume that $w=w'x'$ for some $x'\in X$, and set $R=3R_{r_0}$.
    First, we assume $n=0$ then $T^w_0=w'X$. We decompose $\pi^{-1}(T_m^w\setminus w'X)$ as
    \begin{align*}
        \pi^{-1}(T_m^w\setminus (wX))=\left(N_R(w'X)\setminus w'X \right)\cup \left( \pi^{-1}(T_m^w)\setminus N_R(w'X)\right).
    \end{align*}
    See the Figure~\ref{fig:decomposition_of_S}.
    Here $N_R(w'X)$ is the $R$-neighborhood of the coset $w'X$ in $\pi^{-1}(T_m^w)$. If we denote by $A_0^w:=\left(N_R(w'X)\setminus w'X \right)$, then $A_0^w$ can be written as the disjoint union $$A_0^w=\bigsqcup_{y\in Y} B_{\pi^{-1}(T_{m-1}^{wy})}(we_Y;R).$$
    Note that the map
    \begin{align*}
        \left(N_R(w'X)\right)\setminus w'X =\bigsqcup_{y\in Y} B_{\pi^{-1}(T_{m-1}^{wy})}(we_Y;R) \to w'X:~ wy\mapsto w
    \end{align*}
    is a coarse equivalence with uniform distortion independent of $w$. Therefore, $\left\{A_0^w\right\}_w$ satisfies the quantitative coarse Baum--Connes conjecture.
    
    On the other hand, $\pi^{-1}(T_m^w)\setminus w'X$ is the disjoint union $\bigsqcup_{x} \pi^{-1}(T_{m-1}^{wx})$. For any $w_i\in \pi^{-1}(T_{m-1}^{wx_i})$, $(i=0,1)$ with distinct $x_0$ and $x_1$ in $X$, we have
    \begin{align*}
        d(w_0,w_1)=d(w_0,w'x_0e_Y)+d_X(x_0,x_1)+d(w'x_1e_Y,w_1).
    \end{align*}
    Therefore, by denoting $A_{1,x}^w:=\pi^{-1}(T_{m-1}^{w'x})\setminus B_{\pi^{-1}(T_{m-1}^{w'x})}(w'xe_Y;R)$, we have $d\left(A_{1,x_0}^w,A_{1,x_1}^w\right)> 2R$ and so
    \begin{align*}
        \pi^{-1}(T_m^w)\setminus N_R(w'X)=\bigsqcup_g A_{1,g}^w
    \end{align*}
     is a $2R$-separated disjoint union.
    For every $w\in W$ and $x\in X$, $A_{1,x}^w$ is coarsely equivalent to $\pi^{-1}(T_{m-1}^{w'x})$ with the uniform distortion independent of $w$ and $x$. Therefore, the family $\{A_{1,x}^w\}_{w,x}$ uniformly satisfies the quantitative coarse Baum--Connes conjecture by Lemma~\ref{subtree}. Next, we consider intersections $ \left\{ N_s(A_0^w) \cap N_s(A_{1,x}^w)\right\}_{w,x}$ for $s\leq R_{r_0}$. Notice that $N_s(A_0^w) \cap N_s(A_{1,x}^w)$ is an annulus centered at $wxe_Y$ with radius between $R-s$ to $R+s$. Therefore the metric family $ \left\{ N_s(A_0^w) \cap N_s(A_{1,x}^w)\right\}_{w,x}$ is uniformly bounded.  
    Now the assertion follows from Lemma~\ref{one step decomposition}.
    
    Next, we assume that the assertion holds for $n-1$. In this case, we decompose
    \begin{align*}
        \pi^{-1}(S_{m,n}^w)=\left(N_R(\pi^{-1}(T_n^w))\setminus \pi^{-1}(T_n^w)\right)\cup \left(\pi^{-1}(S_{m,n}^w)\setminus N_R(\pi^{-1}(T_n^w)) \right).
    \end{align*}
    Now, by the same argument as the case of $n=0$, $C_0^w:=N_R(\pi^{-1}(T_n^w))\setminus \pi^{-1}(T_n^w)$ is coarsely equivalent to $\pi^{-1}(T_n^w\setminus T_{n-1}^w)$ uniformly on $w$. Therefore the family $\{C_0^w\}_w$ satisfies the quantitative coarse Baum--Connes conjecture by the inductive assumption for the case $m=1$. On the other hand, again as well as the case for $n=0$, $\pi^{-1}(S_{m,n}^w)\setminus N_R(\pi^{-1}(T_n^w)) $ can be written as a $2R$-separated disjoint union
    \begin{align*}
        \pi^{-1}(S_{m,n}^w)\setminus N_R(\pi^{-1}(T_n^w))  =\bigsqcup_{v\in T_n^w\setminus T_{n-1}^w} C_{1,v}^w ,
    \end{align*}
    where $C_{1,v}^w:=\pi^{-1}(T_m^{v})\setminus B_{\pi^{-1}(T_m^{v})}(v;R)$
     and the metric family $\{C_{1,w'}^w\}_{w,w'}$ satisfies the quantitative coarse Baum--Connes conjecture, since $C_{1,v}^w$ is coarsely equivalent to $\pi^{-1}(T_m^{v})$ uniformly on $w$ and $v$. Again for every $s\leq R_{r_0},$ the intersections $\{N_s(C_0^w)\cap N_s(C_{1,v}^w)\}_{w,v}$ is uniformly bounded. The assertion follows from Lemma~\ref{one step decomposition}.
\end{proof}

\begin{figure}[h]
\centering
\includegraphics[width=0.8\textwidth]{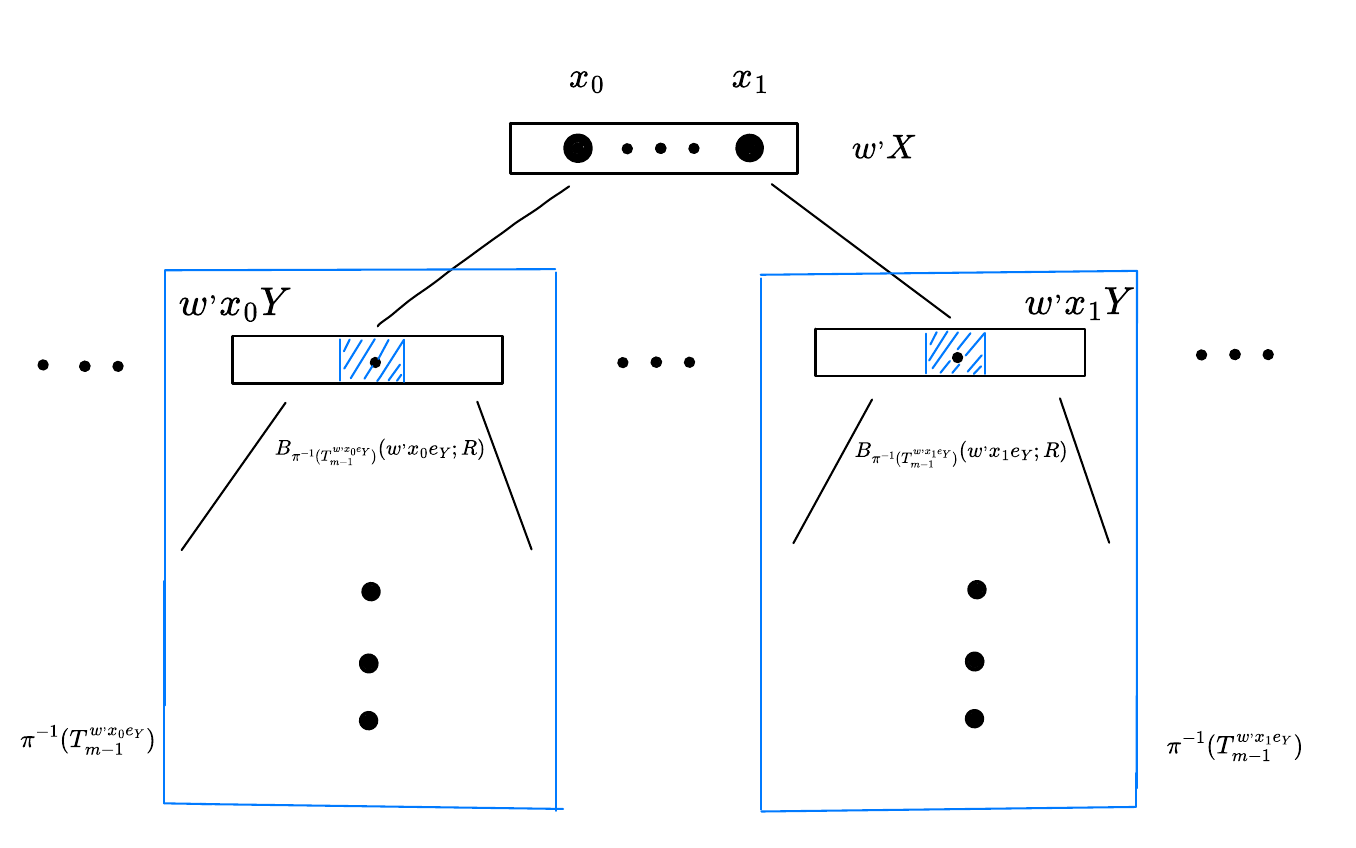}
\caption{The decomposition $\pi^{-1}(T_m^{w})=\left(N_r(w'X)\setminus w'X \right)\cup \left( \pi^{-1}(T_m)\setminus N_r(w'X)\right)$.}
\label{fig:decomposition_of_S}
\end{figure}

Now, we are ready to prove our main result. 

\begin{theorem}

 If $X$ and $Y$ are uniformly locally finite metric spaces which satisfy the quantitative coarse Baum--Connes conjecture, then their free product $X*Y$ satisfies the quantitative coarse Baum--Connes conjecture, i.e.~for all $k>0$ and $r>0$ the map
    \begin{align*}
        K^{\varepsilon,r}_{*}(C^*_{L,0}(P_k(X*Y))) \to K^{\varepsilon,r'}_{*}(C^*_{L,0}(P_{k'}(X*Y)))
    \end{align*}
    is the zero map for some $k'>k$, $r'\geq r$ and for all sufficiently small $\varepsilon$.
\end{theorem}

\begin{proof}
We set $R:=3R_r$ and decompose the Bass--Serre tree $T$ into
    \begin{align*}
        T=\bigsqcup C_{0,j} \cup \bigsqcup C_{1,j'}
    \end{align*}
    as follows. First, we decompose the tree into annuli
    \begin{align*}
        A_{n}:=\{s\in T:nR\leq d_T(s,\phi)< (n+1)R\}.
    \end{align*}
 Then, we further decompose each $A_n$ into the equivalence classes $\{B_{n,j}\}_j$ according to the equivalence relation $\sim$ defined by
    \begin{align*}
        x\sim y \iff \frac{d_T(x,\phi)+d_T(y,\phi)-d_T(x,y)}{2}\geq \left(n-\frac{1}{2}\right)R.
    \end{align*}
    By arranging $\{B_{n,j}\}$, we obtain metric families $\{C_{0,j}\}_j$ and $\{C_{1,j'}\}_{j'}$ as follows
    \begin{align*}
        \{C_{0,j}\}_j&:=\{B_{n,j}:~n=0,2,4,\cdots, ~j=1,2,\cdots\}\\
        \{C_{1,j'}\}_{j'}&:=\{B_{n,j}:~n=1,3,5,\cdots, ~j=1,2,\cdots\}.
    \end{align*}
    Then the metric families $\{C_{0,j}\}_j$ and $\{C_{1,j'}\}_{j'}$ are $R$-separated and each $C_{i,j}$ has the form $C_{i,j}= T^{w_{i,j}}_{n+m}\setminus T^{w_{i,j}}_{n}$, where $m=R$ and $n=\frac{1}{2}R$. Thus the metric families $\{\pi^{-1}(C_{0,j})\}_j$ and $\{\pi^{-1}(C_{1,j'})\}_{j'}$ satisfies the quantitative coarse Baum--Connes conjecture by the previous lemma. Again by Lemma~\ref{one step decomposition}, it suffices to show that the metric family $\{N_{s}(\pi^{-1}(C_{0,j}))\cap N_{s}(\pi^{-1}(C_{1,j'}))\}_{j,j'}$ satisfies the quantitative coarse Baum--Connes conjecture for all $s\leq \frac{R}{3}$. The intersection $N_{s}(\pi^{-1}(C_{0,j}))\cap N_{s}(\pi^{-1}(C_{1,j'}))$ is nonempty only if there exists an $n$ such that one of $C_{0,j}$ and $C_{1,j'}$ is contained in $A_n$ and the other is contained in $A_{n+1}$ for some $n$.
    If $C_{0,j}\subset A_n$ and $C_{1,j'}\subset A_{n+1}$, then 
    \begin{align*}
        N_{s}(\pi^{-1}(C_{0,j}))\cap N_{s}(\pi^{-1}(C_{1,j'})) =\bigcup_w B_{X*Y}(w;s),
    \end{align*}
    where $w$ runs over the set
    \begin{align}\label{set of w}
    \left\{ w\in X*Y : \pi(w)\in C_{1,j'} \text{ and } d_T(\pi(w),A_n)= 1 \right\}.
\end{align}
    The space $\bigcup_w B_{X*Y}(w;s)$ is uniformly coarsely equivalent to the set \eqref{set of w} with distortions independent of $j,j'$. And the set \eqref{set of w} is equal to $\pi^{-1}(T_{n+m}^{w'}\setminus T_{n+m-1}^{w'})$
    for some $w'\in X*Y$ corresponding to the vertex at the base of $C_{0,j}$.
    Thus the family $\{N_{s}(\pi^{-1}(C_{0,j}))\cap N_{s}(\pi^{-1}(C_{1,j'}))\}_{j,j'}$ satisfies the quantitative coarse Baum--Connes conjecture.
\end{proof}



\bibliographystyle{alpha}
\bibliography{ref}



\end{document}